%% file: main.tex
\documentclass[12pt]{article}

\usepackage[T1]{fontenc}
\usepackage[utf8]{inputenc}
\usepackage{lmodern}
\usepackage[margin=1in]{geometry}
\usepackage{amsmath,amssymb,amsthm,mathtools}
\usepackage{booktabs}
\usepackage{microtype}
\usepackage{setspace}
\usepackage[round,authoryear]{natbib}
\usepackage[hidelinks]{hyperref}
\usepackage[nameinlink,noabbrev]{cleveref}

\input{macros}

\newtheorem{theorem}{Theorem}[section]
\newtheorem{proposition}[theorem]{Proposition}
\newtheorem{lemma}[theorem]{Lemma}
\newtheorem{corollary}[theorem]{Corollary}
\theoremstyle{definition}

\theoremstyle{remark}
\newtheorem{remark}[theorem]{Remark}

\numberwithin{equation}{section}
\allowdisplaybreaks
\title{Directional Optimal Sub-Gamma Scales for Infinitely Divisible Laws}
\author{Yichuan Chen\textsuperscript{1} and Xin Wang\textsuperscript{1,*}\\
\small \textsuperscript{1}School of Art and Design, Zhejiang Sci-Tech University,
Hangzhou 310018, China}
\date{}

\begin{document}
\maketitle

\begin{center}
\small
\textbf{Running title:} Directional Sub-Gamma Scales\\
\textbf{Corresponding author:} Xin Wang; email: wxcaup@zstu.edu.cn\\
\textbf{First author email:} Yichuan Chen; email: smallboychen@gmail.com
\end{center}

\begin{abstract}
Fixing the quadratic proxy in a sub-gamma bound at the true variance leaves a
scale to optimize, and a two-sided infinitely divisible law generally requires
different scales in the two directions. For a centered law with finite nonzero
variance, we normalize its Kolmogorov canonical measure and multiply the
resulting variable by an independent $\operatorname{Beta}(1,2)$ variable. The
signed remainder obtained in this way gives exact variational formulas for the
right and left scales. We prove that a directional scale vanishes exactly when
the L\'evy measure has no jumps in that direction, establish reflection,
scaling, convolution, L\'evy-time, and opposite-jump perturbation rules, and
recover the L\'evy triplet from the remainder law. The formulas give the two
Gamma scales for bilateral Gamma laws and, for centered Skellam laws, the exact
transition points $p_{\pm}=(2\pm\sqrt{3})/4$ between local and interior control;
when positive jumps are rare, the right scale is asymptotic to
$1/\log(1/p)$. The pair therefore records jump direction and the mechanism
that controls the variance-exact sub-gamma pole.
\end{abstract}



\input{sections/01_introduction}
\input{sections/02_remainder}
\input{sections/03_scales}
\input{sections/04_operations}
\input{sections/05_exact_families}
\input{sections/06_skellam}
\input{sections/07_rare_jumps}
\input{sections/08_discussion}

\section*{Statements and Declarations}

\paragraph{Funding.}
The authors received no specific funding for this work.

\paragraph{Competing interests.}
The authors declare no competing interests.

\paragraph{Data availability.}
No datasets were generated or analyzed. The symbolic, numerical, and formal
verification files are supplied with the manuscript source.

\paragraph{Author contributions.}
Y.C. developed the theory, proved the results, performed the symbolic,
numerical, and formal checks, and drafted the manuscript. X.W. supervised the
work and reviewed and revised the manuscript. Both authors approved the final
version.

\paragraph{Use of generative artificial intelligence.}
OpenAI Codex (GPT-5, accessed August 2026) assisted with English translation
and the preparation of Lean formalization files. The authors verified all
mathematical statements and take responsibility for the final manuscript.

\bibliographystyle{plainnat}
\bibliography{references}

\end{document}

%% file: macros.tex
\newcommand{\E}{\mathbb{E}}
\newcommand{\Pp}{\mathbb{P}}
\newcommand{\R}{\mathbb{R}}

\newcommand{\Var}{\operatorname{Var}}
\newcommand{\Law}{\mathcal{L}}
\newcommand{\cplus}{c_{+}}
\newcommand{\cminus}{c_{-}}
\newcommand{\tplus}{\tau_{+}}
\newcommand{\tminus}{\tau_{-}}
\newcommand{\one}{\mathbf{1}}
\newcommand{\dd}{\,\mathrm{d}}
\newcommand{\pos}[1]{\left(#1\right)_{+}}

%% file: sections/01_introduction.tex
\section{Introduction}
\label{sec:introduction}

A centered random variable $X$ is right sub-gamma with variance factor $v>0$
and scale $c\geq 0$ if
\begin{equation}
  \log \E e^{tX}\leq \frac{v t^2}{2(1-ct)},
  \qquad 0\leq t<1/c,
  \label{eq:subgamma}
\end{equation}
where $1/0=\infty$. This form leads to the usual Bernstein tail bound and is a
standard tool in concentration theory \citep{boucheron2013concentration}. The
parameters in \eqref{eq:subgamma} are not intrinsic unless one of them is fixed.
Increasing either parameter preserves the inequality, so a convenient pair can
be far from sharp.

We fix $v$ at the true variance $V=\Var(X)$. This is the smallest possible
quadratic coefficient in a neighborhood of the origin. The remaining question
is global: what is the smallest compatible scale? For a nonsymmetric law, the
answer can differ after reflection, and we write the two optimal values as
$\cplus(X)$ and $\cminus(X)=\cplus(-X)$. We determine this pair for general
centered infinitely divisible laws with finite nonzero variance and describe
what it records about their L\'evy spectra.

Deviation inequalities for infinitely divisible laws are often expressed
directly through the L\'evy measure. General results of this kind include
\citet{houdre2002remarks}, concentration for norms of infinitely divisible
vectors \citep{houdre2008concentration}, and compound Poisson inequalities
\citep{kontoyiannis2006measure}. Stein and covariance representations provide a
different route to distributional identities and variance bounds
\citep{arras2019stein,barman2025covariance}. These methods yield useful controls,
but they do not by themselves determine the smallest scale in
\eqref{eq:subgamma} when its quadratic coefficient is fixed at $V$.

Optimal proxy problems have been studied for other envelopes. Examples include
optimal sub-Gaussian variance proxies for bounded and truncated laws
\citep{arbel2020strict,barreto2026optimal}. Most recently,
\citet{leskela2026subpoisson} developed optimal one-sided and two-sided
sub-Poisson variance proxies and their closure properties. Their examples
include the Skellam law, whose optimal sub-Poisson proxy is its variance. The
present problem keeps the variance fixed and optimizes the pole of a different
envelope. The distinction matters: in the Skellam family it produces a phase
transition that is absent from the sub-Poisson calculation.

We use the finite-variance canonical representation of an infinitely divisible
law \citep{sato1999levy,steutel2004infinite}. If $a$ is the
Gaussian variance and $\nu$ is the L\'evy measure, then
\begin{equation}
 K_X(t)=\log\E e^{tX}
 =\frac{a t^2}{2}+\int_{\R\setminus\{0\}}
      \bigl(e^{tx}-1-tx\bigr)\,\nu(\dd x),
 \qquad
 V=a+\int x^2\,\nu(\dd x).
 \label{eq:intro-lk}
\end{equation}
The finite measure $a\delta_0+x^2\nu(\dd x)$ is the Kolmogorov canonical
measure. Normalize it by $V$, let $Y_X$ have the resulting probability law, and
let $B\sim\operatorname{Beta}(1,2)$ be independent of $Y_X$. The signed
canonical remainder
\begin{equation}
 R_X=B Y_X
 \label{eq:intro-remainder}
\end{equation}
satisfies
\begin{equation}
 K_X(t)=\frac{Vt^2}{2}M_{R_X}(t),
 \qquad M_{R_X}(t)=\E e^{tR_X},
 \label{eq:intro-factor}
\end{equation}
with extended values allowed. The canonical representation and the elementary
Beta integral behind \eqref{eq:intro-factor} are not presented as new. A
related manuscript under review treats the spectrally positive case, where
$R_X$ is nonnegative and can be compared with an exponential law through the
M-class of \citet{klar2003characterizations} \citep{chen2026spectrally}. The
present paper does not use that manuscript as a premise: every definition and
result needed below is stated and proved here. In the present setting $R_X$ is
signed, so its two transforms must be optimized separately and the nonnegative
order argument is no longer available.

For signed remainders, the density of $R_X$ on each half-line is an integrated
tail of the corresponding part of $\nu$, and its distributional second
derivative recovers that part of the L\'evy measure. In addition,
\eqref{eq:intro-factor} gives the exact formulas
\begin{equation}
 \cplus(X)=\max\left\{0,\sup_{t>0}
   \frac{1-M_{R_X}(t)^{-1}}{t}\right\},
 \qquad
 \cminus(X)=\max\left\{0,\sup_{t>0}
   \frac{1-M_{R_X}(-t)^{-1}}{t}\right\}.
 \label{eq:intro-variational}
\end{equation}
These values have an exact spectral zero pattern:
\begin{equation}
 \cplus(X)=0\quad\Longleftrightarrow\quad \nu((0,\infty))=0,
 \qquad
 \cminus(X)=0\quad\Longleftrightarrow\quad \nu((-\infty,0))=0.
 \label{eq:intro-zero}
\end{equation}
Thus the pair distinguishes positive jumps, negative jumps, two-sided jumps,
and the purely Gaussian case.

We also show that adding an independent infinitely divisible variable with no
positive jumps cannot increase $\cplus$. Although the total variance changes,
the signed remainder of a sum is a variance-weighted mixture, which makes the
monotonicity exact. If the original scale is forced by its positive MGF
boundary, equality is preserved. Reflection gives the left counterpart.

The bilateral Gamma and Skellam families realize two different optimizing
mechanisms. For a centered bilateral Gamma law, the right and left scales are
precisely the two Gamma scale parameters. For the centered Skellam family with L\'evy measure
$\lambda[p\delta_1+(1-p)\delta_{-1}]$, the local cumulant bound controls the
right scale exactly only when
\begin{equation}
 p\geq p_+=\frac{2+\sqrt{3}}{4}.
\end{equation}
For $0<p<p_+$, the optimum is attained at a positive interior argument. The
left transition occurs at $p_-=(2-\sqrt{3})/4$. When $p\downarrow0$, any right
optimizer $t_p$ satisfies
\begin{equation}
 t_p=\log(1/p)+3\log\log(1/p)-\log 2+o(1),
 \qquad
 \cplus(p)\sim\frac{1}{\log(1/p)}.
\end{equation}
The scale therefore vanishes at $p=0$ but approaches zero only logarithmically
when positive jumps become rare.

The paper is organized as follows.
\Cref{sec:remainder} gives the signed factorization, density, and recovery
formula. \Cref{sec:scales} proves the variational characterization and its local,
boundary, and support consequences. \Cref{sec:operations} treats structural
operations and opposite-jump perturbations. \Cref{sec:families} records the
equality laws and solves the bilateral Gamma family. \Cref{sec:skellam} proves
the Skellam transition, and \Cref{sec:rare} establishes the rare-jump
asymptotics.

%% file: sections/02_remainder.tex
\section{The signed canonical remainder}
\label{sec:remainder}

Throughout, $X$ is centered, infinitely divisible, and nondegenerate, with
finite variance $V>0$. We use the representation \eqref{eq:intro-lk}. The
integrand $e^{tx}-1-tx$ is nonnegative, so $K_X(t)$ is well defined in
$[0,\infty]$ for every real $t$. At points where it is finite, it is the usual
cumulant generating function.

Define the finite canonical measure
\begin{equation}
 H_X(\dd x)=a\delta_0(\dd x)+x^2\nu(\dd x).
 \label{eq:canonical-H}
\end{equation}
Its total mass is $V$. Let $Y_X\sim H_X/V$, take
$B\sim\operatorname{Beta}(1,2)$ independently, and set $R_X=BY_X$. The density
of $B$ is $2(1-b)$ on $(0,1)$.

\begin{theorem}[Signed canonical factorization]
\label{thm:factorization}
For every real $t$,
\begin{equation}
 K_X(t)=\frac{Vt^2}{2}M_{R_X}(t),
 \label{eq:factorization}
\end{equation}
where both sides are interpreted by continuity at $t=0$ and extended values are
allowed away from zero. The law of $R_X$ has an atom of mass $a/V$ at zero. On
the two open half-lines it has the continuous densities
\begin{align}
 q_X(r)&=\frac{2}{V}\int_{(r,\infty)}(x-r)\,\nu(\dd x),
 &&r>0, \label{eq:q-positive}\\
 q_X(r)&=\frac{2}{V}\int_{(-\infty,r)}(r-x)\,\nu(\dd x),
 &&r<0. \label{eq:q-negative}
\end{align}
The pair $(V,\Law(R_X))$ determines $a$, $\nu$, and hence the centered law of
$X$. More precisely,
\begin{equation}
 \nu=\frac{V}{2}q_X''
 \label{eq:recovery}
\end{equation}
in the sense of distributions on each open half-line.
\end{theorem}

\begin{proof}
For $z\neq0$, direct integration gives
\begin{equation}
 \E e^{zB}=2\int_0^1(1-b)e^{zb}\,\dd b
 =\frac{2(e^z-1-z)}{z^2}.
 \label{eq:beta-identity}
\end{equation}
The value at $z=0$ is one by continuity. Conditioning first on $Y_X$ and using
\eqref{eq:canonical-H},
\begin{align*}
 \frac{Vt^2}{2}\E e^{tBY_X}
 &=\frac{t^2}{2}\int_{\R}\E e^{tBx}\,H_X(\dd x)\\
 &=\frac{at^2}{2}+
   \int_{\R\setminus\{0\}}(e^{tx}-1-tx)\,\nu(\dd x),
\end{align*}
which proves \eqref{eq:factorization}.

If $x>0$, the conditional density of $Bx$ at $r\in(0,x)$ is
$2(x-r)/x^2$. Multiplication by the $x^2\nu(\dd x)/V$ part of the law of
$Y_X$ and integration over $x>r$ gives \eqref{eq:q-positive}. If $x<0$, the
conditional density at $r\in(x,0)$ is $2(r-x)/x^2$, which gives
\eqref{eq:q-negative}. The Gaussian part of $H_X$ yields the stated atom at
zero, and no other part produces an atom there.

It remains to justify recovery. On $(0,\infty)$, the function
$Vq_X/2$ is the integrated upper tail of $\nu$. Its first distributional
derivative is $-\nu((r,\infty))$, and its second derivative is the restriction
of $\nu$ to that half-line. The same calculation on $(-\infty,0)$ starts from
the integrated lower tail and has the same second derivative. The mass at zero
of $R_X$ gives $a/V$. Thus $V$ and $\Law(R_X)$ recover $a$ and $\nu$. In the
centered representation these determine $K_X$, so they determine the law of
$X$.
\end{proof}

\begin{remark}
The reconstruction statement is stronger than what is needed for concentration.
It shows that the signed remainder is not a lossy summary once its natural scale
$V$ is retained. The pair $(\cplus,\cminus)$ is much coarser, but its zero
pattern will still recover the support directions of $\nu$.
\end{remark}

The factorization also gives a useful mixture rule. If $X_1$ and $X_2$ are
independent variables in the standing class, with variances $V_1,V_2$, then
the canonical measure of $X_1+X_2$ is $H_{X_1}+H_{X_2}$. Therefore
\begin{equation}
 \Law(R_{X_1+X_2})
 =\frac{V_1}{V_1+V_2}\Law(R_{X_1})
  +\frac{V_2}{V_1+V_2}\Law(R_{X_2}),
 \label{eq:remainder-mixture}
\end{equation}
where the right side denotes a probability mixture. This identity will drive
the convolution and perturbation results in \Cref{sec:operations}.

%% file: sections/03_scales.tex
\section{Exact directional scales}
\label{sec:scales}

Define
\begin{align}
 \cplus(X)=\inf\Bigl\{c\geq0:\;&K_X(t)\leq
  \frac{Vt^2}{2(1-ct)}\ \text{for every }0\leq t<1/c\Bigr\},
 \label{eq:def-cplus}\\
 \cminus(X)&=\cplus(-X).
 \label{eq:def-cminus}
\end{align}
The infimum of the empty set is $\infty$. Let
\begin{equation}
 \tplus(X)=\sup\{t>0:K_X(t)<\infty\},
 \qquad
 \tminus(X)=\tplus(-X),
 \label{eq:abscissae}
\end{equation}
with values in $[0,\infty]$. We use $1/\infty=0$ and $1/0=\infty$.

\begin{theorem}[Variational characterization]
\label{thm:variational}
For the signed canonical remainder of \Cref{thm:factorization},
\begin{align}
 \cplus(X)&=\max\left\{0,\sup_{t>0}\Psi_+(t)\right\},
 &\Psi_+(t)&=\frac{1-M_{R_X}(t)^{-1}}{t},
 \label{eq:cplus-variation}\\
 \cminus(X)&=\max\left\{0,\sup_{t>0}\Psi_-(t)\right\},
 &\Psi_-(t)&=\frac{1-M_{R_X}(-t)^{-1}}{t}.
 \label{eq:cminus-variation}
\end{align}
Here $1/\infty=0$. If the value on the right is finite, it is itself a feasible
scale. Moreover,
\begin{equation}
 \cplus(X)<\infty\quad\Longleftrightarrow\quad\tplus(X)>0,
 \qquad
 \cminus(X)<\infty\quad\Longleftrightarrow\quad\tminus(X)>0.
 \label{eq:finite-iff-mgf}
\end{equation}
\end{theorem}

\begin{proof}
Fix $c\geq0$ and $0<t<1/c$. By \eqref{eq:factorization}, the sub-gamma
inequality is equivalent to
\begin{equation}
 M_{R_X}(t)\leq\frac{1}{1-ct}.
 \label{eq:M-exponential-envelope}
\end{equation}
When the left side is finite, this is equivalent to $c\geq\Psi_+(t)$. If the
left side is infinite, then $\Psi_+(t)=1/t>c$, so the same equivalence holds in
the extended sense. For $t\geq1/c$, the elementary bound
$\Psi_+(t)\leq1/t\leq c$ applies without using feasibility. Thus any feasible
$c$ must dominate zero and every $\Psi_+(t)$.

Conversely, let $c$ be the right side of \eqref{eq:cplus-variation} and suppose
it is finite. If $0<t<1/c$, then $M_{R_X}(t)$ cannot be infinite, because that
would give $\Psi_+(t)=1/t>c$. The inequality $c\geq\Psi_+(t)$ can now be
rearranged to \eqref{eq:M-exponential-envelope}. Thus $c$ is feasible and the
formula for $\cplus$ follows. Reflection proves \eqref{eq:cminus-variation}.

By \eqref{eq:factorization}, $M_{R_X}(t)<\infty$ at a nonzero $t$ exactly when
$K_X(t)<\infty$. If no positive exponential moment exists, then
$M_{R_X}(t)=\infty$ and $\Psi_+(t)=1/t$ for every $t>0$, so $\cplus=\infty$.
If $M_{R_X}(t_0)<\infty$ for some $t_0>0$, convexity of the MGF gives finiteness
and local boundedness on $[0,t_0)$. Hence $\Psi_+$ is bounded there, while
$\Psi_+(t)\leq1/t\leq1/t_0$ for $t\geq t_0$. This proves the first equivalence
in \eqref{eq:finite-iff-mgf}; the second follows by reflection.
\end{proof}

The following spectral consequence is exact and does not require higher
moments.

\begin{theorem}[Spectral zero pattern]
\label{thm:zero-pattern}
The directional scales satisfy
\begin{equation}
 \cplus(X)=0\quad\Longleftrightarrow\quad\nu((0,\infty))=0,
 \qquad
 \cminus(X)=0\quad\Longleftrightarrow\quad\nu((-\infty,0))=0.
 \label{eq:zero-pattern}
\end{equation}
Consequently, $\cplus(X)=\cminus(X)=0$ if and only if $X$ is Gaussian with
variance $V$.
\end{theorem}

\begin{proof}
If $\nu((0,\infty))=0$, then $Y_X\leq0$ and hence $R_X\leq0$ almost surely.
Thus $M_{R_X}(t)\leq1$ for $t>0$, so \eqref{eq:cplus-variation} gives
$\cplus=0$. Conversely, if the L\'evy measure assigns positive mass to
$(0,\infty)$, then $\Pp(R_X>0)>0$. Monotone convergence gives
$M_{R_X}(t)\to\infty$ as $t\to\infty$, possibly through infinite values. Hence
$M_{R_X}(t)>1$ for some $t$ and $\cplus>0$. This proves the first equivalence;
reflection proves the second. Both scales vanish exactly when $\nu=0$. The
centered L\'evy representation then reduces to $K_X(t)=Vt^2/2$.
\end{proof}

Two endpoint obstructions are useful below. Write $\kappa_j(X)$ for the
$j$th cumulant when it exists.

\begin{proposition}[Cumulants and endpoint behavior]
\label{prop:endpoints}
For every integer $n\geq1$ such that
$\int |x|^{n+2}\nu(\dd x)<\infty$,
\begin{equation}
 \E R_X^n=\frac{2\kappa_{n+2}(X)}{(n+1)(n+2)V}.
 \label{eq:R-moments}
\end{equation}
In particular, when the third cumulant is finite,
\begin{equation}
 \cplus(X)\geq\pos{\frac{\kappa_3}{3V}},
 \qquad
 \cminus(X)\geq\pos{-\frac{\kappa_3}{3V}}.
 \label{eq:local-bounds}
\end{equation}
The MGF boundaries give
\begin{equation}
 \cplus(X)\geq\frac{1}{\tplus(X)},
 \qquad
 \cminus(X)\geq\frac{1}{\tminus(X)}.
 \label{eq:boundary-bounds}
\end{equation}
If $\kappa_4$ is finite, then both directional objectives have the same
one-sided derivative at the origin:
\begin{equation}
 \Psi_+'(0+)=\Psi_-'(0+)
 =\frac{\kappa_4}{12V}-\frac{\kappa_3^2}{9V^2}.
 \label{eq:endpoint-derivative}
\end{equation}
\end{proposition}

\begin{proof}
Independence and the density of $B$ give
\begin{equation*}
 \E B^n=2\int_0^1b^n(1-b)\,\dd b
 =\frac{2}{(n+1)(n+2)}.
\end{equation*}
The nonzero part of the law of $Y_X$ then yields
\begin{equation*}
 \E R_X^n=\E B^n\frac{1}{V}\int x^{n+2}\nu(\dd x).
\end{equation*}
For $n+2\geq3$, the last integral is $\kappa_{n+2}(X)$, which proves
\eqref{eq:R-moments}.

Let $m_1=\E R_X=\kappa_3/(3V)$. Expanding the MGF at zero gives
$\Psi_+(t)\to m_1$ and $\Psi_-(t)\to-m_1$. The variational formulas imply
\eqref{eq:local-bounds}. If $c<1/\tplus$, then the interval $(0,1/c)$ extends
beyond the positive MGF domain, where the left side of \eqref{eq:subgamma} is
infinite and the right side is finite. This proves the first boundary bound;
reflection gives the second.

Finally, set $m_2=\E R_X^2=\kappa_4/(6V)$. From
\begin{equation*}
 M_{R_X}(t)=1+m_1t+\frac{m_2t^2}{2}+o(t^2)
\end{equation*}
we obtain
\begin{equation*}
 \Psi_+(t)=m_1+\left(\frac{m_2}{2}-m_1^2\right)t+o(t).
\end{equation*}
Replacing $t$ by $-t$ changes $m_1$ to $-m_1$ but leaves the coefficient of
$t$ unchanged. Substitution of $m_1$ and $m_2$ gives
\eqref{eq:endpoint-derivative}.
\end{proof}

The derivative separates local from interior control. We say that the right
scale is locally controlled when the supremum in \eqref{eq:cplus-variation} is
given by its limit at zero, after taking the maximum with zero. Boundary and
interior control have the analogous meanings.

\begin{corollary}[An interior-control criterion]
\label{cor:interior}
Suppose that the L\'evy measure is supported in a bounded interval and charges
both half-lines. If
\begin{equation}
 3V\kappa_4>4\kappa_3^2,
 \label{eq:interior-criterion}
\end{equation}
then both directional scales are attained at positive finite arguments. No
uniqueness is asserted.
\end{corollary}

\begin{proof}
Condition \eqref{eq:interior-criterion} makes the derivative in
\eqref{eq:endpoint-derivative} positive, so neither objective is maximized at
zero. Bounded L\'evy support makes both MGFs entire. Since both jump directions
are present, \Cref{thm:zero-pattern} gives positive scales. Moreover,
$\Psi_+(t)\to0$ and $\Psi_-(t)\to0$ as $t\to\infty$: in each direction the
corresponding MGF tends to infinity, while the prefactor $1/t$ tends to zero.
Continuity on $(0,\infty)$ now forces each positive supremum to be attained in
the interior.
\end{proof}

%% file: sections/04_operations.tex
\section{Structural operations and directional perturbations}
\label{sec:operations}

The canonical remainder turns the basic operations into scaling and mixture
identities. The statements below include extended values when an exponential
moment is absent.

\begin{proposition}[Reflection, scaling, and L\'evy time]
\label{prop:basic-operations}
For $b>0$ and $s>0$,
\begin{align}
 \cplus(-X)&=\cminus(X), & \cminus(-X)&=\cplus(X),
 \label{eq:reflection-scales}\\
 \cplus(bX)&=b\cplus(X), & \cminus(bX)&=b\cminus(X).
 \label{eq:scaling-scales}
\end{align}
If $X_s$ is the time-$s$ marginal of a L\'evy process whose time-one marginal
is $X$, then
\begin{equation}
 \cplus(X_s)=\cplus(X),
 \qquad
 \cminus(X_s)=\cminus(X).
 \label{eq:levy-time}
\end{equation}
\end{proposition}

\begin{proof}
Reflection sends $R_X$ to $-R_X$, which proves \eqref{eq:reflection-scales}.
Positive scaling sends the canonical measure to the pushforward under
$x\mapsto bx$, after multiplication by $b^2$, and sends $V$ to $b^2V$.
Therefore $R_{bX}$ has the law of $bR_X$. Substitution in
\eqref{eq:cplus-variation} gives \eqref{eq:scaling-scales}. At L\'evy time $s$,
both the canonical measure and its total mass are multiplied by $s$. The
normalized law of $Y_X$, hence that of $R_X$, is unchanged. This proves
\eqref{eq:levy-time}.
\end{proof}

\begin{proposition}[Independent sums]
\label{prop:convolution}
If $X_1$ and $X_2$ are independent centered infinitely divisible variables with
finite positive variances, then
\begin{equation}
 \cplus(X_1+X_2)\leq\max\{\cplus(X_1),\cplus(X_2)\},
 \qquad
 \cminus(X_1+X_2)\leq\max\{\cminus(X_1),\cminus(X_2)\}.
 \label{eq:convolution-scales}
\end{equation}
\end{proposition}

\begin{proof}
Let $c=\max\{\cplus(X_1),\cplus(X_2)\}$. For $0<t<1/c$, each remainder MGF
is bounded by $(1-ct)^{-1}$. The mixture identity
\eqref{eq:remainder-mixture} gives the same bound for
$M_{R_{X_1+X_2}}(t)$. Hence $c$ is feasible for the sum. Reflection gives the
left inequality.
\end{proof}

The maximum bound is standard for sub-gamma sums, but the signed remainder also
gives a sharper monotonicity when the added jumps point against the direction
being controlled.

\begin{theorem}[Opposite-jump monotonicity]
\label{thm:opposite-jumps}
Let $X$ and $Z$ be independent variables in the standing class.

\begin{enumerate}
\item If the L\'evy measure of $Z$ has no positive jumps, then
\begin{equation}
 \cplus(X+Z)\leq\cplus(X).
 \label{eq:right-monotonicity}
\end{equation}
\item If the L\'evy measure of $Z$ has no negative jumps, then
\begin{equation}
 \cminus(X+Z)\leq\cminus(X).
 \label{eq:left-monotonicity}
\end{equation}
\end{enumerate}
Suppose in addition that $0<\tplus(X)<\infty$ and
$\cplus(X)=1/\tplus(X)$. Under the assumption in part 1,
\begin{equation}
 \cplus(X+Z)=\cplus(X).
 \label{eq:boundary-robustness}
\end{equation}
The reflected boundary statement holds on the left.
\end{theorem}

\begin{proof}
If $Z$ has no positive jumps, then $R_Z\leq0$ almost surely by
\Cref{thm:zero-pattern}, so $M_{R_Z}(t)\leq1$ for $t>0$. By
\eqref{eq:remainder-mixture}, $M_{R_{X+Z}}(t)$ lies between
$M_{R_X}(t)$ and $M_{R_Z}(t)$. If $M_{R_X}(t)\geq1$, then
$M_{R_{X+Z}}(t)\leq M_{R_X}(t)$ and the map
$m\mapsto(1-m^{-1})/t$ is increasing. If $M_{R_X}(t)<1$, then both mixture
components are at most one and the objective for the sum is nonpositive. In
either case,
\begin{equation*}
 \Psi_{+,X+Z}(t)\leq\max\{0,\Psi_{+,X}(t)\}.
\end{equation*}
Taking suprema proves \eqref{eq:right-monotonicity}. Part 2 follows by
reflection.

A variable with no positive jumps has finite positive exponential moments of
every order. Indeed, in \eqref{eq:intro-lk} its positive-argument integral is
finite because finite variance implies a finite first moment for the large
negative jumps. Hence
$\tplus(X+Z)=\tplus(X)$. The boundary lower bound
\eqref{eq:boundary-bounds} gives
\begin{equation*}
 \cplus(X+Z)\geq1/\tplus(X+Z)=1/\tplus(X)=\cplus(X).
\end{equation*}
Together with \eqref{eq:right-monotonicity}, this proves
\eqref{eq:boundary-robustness}.
\end{proof}

\begin{remark}
The theorem concerns the variance-exact scale, not an envelope with a fixed
external variance proxy. Adding $Z$ changes the true variance from $V_X$ to
$V_X+V_Z$. The conclusion survives because the normalized remainder law changes
by the same variance weights.
\end{remark}

%% file: sections/05_exact_families.tex
\section{Equality laws and bilateral Gamma variables}
\label{sec:families}

The one-sided equality law also appears in the related manuscript under review
\citep{chen2026spectrally}. We include the proof here to keep the paper
self-contained. It also shows that signed jumps create no additional
extremizers.

\begin{proposition}[Equality in the sub-gamma envelope]
\label{prop:equality-laws}
Suppose that, for some finite $c\geq0$, equality holds in
\eqref{eq:subgamma} on a nonempty open interval. If $c=0$, then
$X\sim N(0,V)$. If $c>0$, then $X$ is a centered compound Poisson variable with
L\'evy density
\begin{equation}
 \nu(\dd x)=\frac{V}{2c^3}e^{-x/c}\one_{(0,\infty)}(x)\,\dd x.
 \label{eq:equality-levy}
\end{equation}
Equivalently, its jump intensity is $V/(2c^2)$ and its jumps are exponential
with mean $c$. The left equality laws are the reflections of these variables.
\end{proposition}

\begin{proof}
On the interior of its finite MGF domain, $K_X$ is real analytic. Equality on
an open interval therefore extends to the connected positive domain. By
\eqref{eq:factorization},
\begin{equation}
M_{R_X}(t)=\frac{1}{1-ct}
 \label{eq:equality-remainder-mgf}
\end{equation}
for all sufficiently small $t>0$. Choose an interior point of this interval
and exponentially tilt both sides there. The tilted laws have moment
generating functions in a neighborhood of zero, so uniqueness identifies
$R_X$ with the degenerate law at zero when $c=0$, and with an exponential law
of mean $c$ when $c>0$. In the first case, recovery in
\Cref{thm:factorization} yields $\nu=0$ and $a=V$, so $X$ is Gaussian.

If $c>0$, $R_X$ therefore has density $q_X(r)=c^{-1}e^{-r/c}$ on the positive
half-line, no negative part, and no atom at zero. The recovery formula
\eqref{eq:recovery} gives \eqref{eq:equality-levy}. Its total mass is
$V/(2c^2)$, and normalizing it gives the exponential jump law. Conversely, a
centered compound Poisson variable with this L\'evy measure has
\begin{equation*}
 K_X(t)=\frac{V}{2c^2}\left(\frac{1}{1-ct}-1-ct\right)
 =\frac{Vt^2}{2(1-ct)},
\end{equation*}
so the characterization is exact.
\end{proof}

A useful two-sided exact family is bilateral Gamma. We use the
shape-scale convention. Let $G_+$ and $G_-$ be independent with
\begin{equation*}
 G_+\sim\operatorname{Gamma}(\alpha_+,\beta_+),
 \qquad
 G_-\sim\operatorname{Gamma}(\alpha_-,\beta_-),
\end{equation*}
where every parameter is positive, and define the centered difference
\begin{equation*}
 X=(G_+-\alpha_+\beta_+)-(G_--\alpha_-\beta_-).
\end{equation*}
This is the
bilateral Gamma family of \citet{kuchler2008bilateral}; see also
\citet{barman2025covariance} for its position among infinitely divisible and
variance-gamma models.

\begin{theorem}[Bilateral Gamma scales]
\label{thm:bilateral-gamma}
For the centered bilateral Gamma variable above,
\begin{equation}
 \cplus(X)=\beta_+,
 \qquad
 \cminus(X)=\beta_-.
 \label{eq:bilateral-gamma-scales}
\end{equation}
Both values are controlled by the corresponding MGF boundary.
\end{theorem}

\begin{proof}
The centered positive Gamma component has cumulant generating function
\begin{equation*}
 K_+(t)=\alpha_+\{-\log(1-\beta_+t)-\beta_+t\},
 \qquad t<1/\beta_+.
\end{equation*}
For $0\leq x<1$,
\begin{equation}
 -\log(1-x)-x
 =\sum_{n\geq2}\frac{x^n}{n}
 \leq\frac{x^2}{2(1-x)}.
 \label{eq:gamma-series-bound}
\end{equation}
Thus the centered positive component is right sub-gamma with its true variance
$\alpha_+\beta_+^2$ and scale $\beta_+$. Its positive MGF boundary is
$1/\beta_+$, so \eqref{eq:boundary-bounds} shows that this scale is optimal.

The centered variable $-G_-$ has no positive L\'evy jumps and therefore has
right scale zero. \Cref{thm:opposite-jumps} gives
$\cplus(X)\leq\beta_+$. The positive MGF boundary of $X$ remains
$1/\beta_+$, since the negative Gamma component has all positive exponential
moments. The boundary lower bound gives the reverse inequality. Reflection
proves the formula for $\cminus$.
\end{proof}

In the symmetric case $\alpha_+=\alpha_-$ and $\beta_+=\beta_-$, the third
cumulant vanishes while both scales are positive. More generally, cancellation
can make $\kappa_3=0$ even when the two scales differ. This shows why the local
third-cumulant obstruction in \eqref{eq:local-bounds} cannot identify the
directional pair for two-sided laws.

%% file: sections/06_skellam.tex
\section{The Skellam phase transition}
\label{sec:skellam}

Let $\lambda>0$ and $p\in[0,1]$. Take independent Poisson variables
$N_+\sim\operatorname{Poisson}(\lambda p)$ and
$N_-\sim\operatorname{Poisson}(\lambda(1-p))$, and set
\begin{equation}
 X_{\lambda,p}=N_+-N_- -\lambda(2p-1).
 \label{eq:centered-skellam}
\end{equation}
This is the centered Skellam law \citep{skellam1946frequency}. Its L\'evy
measure and variance are
\begin{equation}
 \nu=\lambda\{p\delta_1+(1-p)\delta_{-1}\},
 \qquad V=\lambda.
 \label{eq:skellam-levy}
\end{equation}
The optimal scales do not depend on $\lambda$ by L\'evy-time invariance.

Put $d=2p-1$. The canonical variable $Y$ takes the values $1$ and $-1$ with
probabilities $p$ and $1-p$. By \eqref{eq:beta-identity}, the right remainder
MGF is
\begin{equation}
 M_p(t)=\frac{2}{t^2}\{\cosh t-1+d(\sinh t-t)\}.
 \label{eq:skellam-M}
\end{equation}
Its series is
\begin{equation}
 M_p(t)=\sum_{n=0}^{\infty}a_n(d)t^n,
 \qquad
 a_n(d)=
 \begin{cases}
  2/(n+2)!,&n\ \text{even},\\
  2d/(n+2)!,&n\ \text{odd}.
 \end{cases}
 \label{eq:skellam-coefficients}
\end{equation}
In particular,
\begin{equation}
 M_p(t)=1+\frac{d}{3}t+\frac{1}{12}t^2
       +\frac{d}{60}t^3+\frac{1}{360}t^4+O(t^5).
 \label{eq:skellam-M-series}
\end{equation}
For
\begin{equation}
 \Psi_p(t)=\frac{1-M_p(t)^{-1}}{t},
 \label{eq:skellam-objective}
\end{equation}
we therefore have
\begin{equation}
 \Psi_p(t)=\frac{d}{3}
 +\left(\frac{1}{12}-\frac{d^2}{9}\right)t+O(t^2).
 \label{eq:skellam-Psi-series}
\end{equation}
The sign change in the linear term gives the candidate thresholds. The next
lemma proves that the candidate is globally correct.

\begin{lemma}[Coefficientwise local regime]
\label{lem:skellam-coefficients}
If $d\geq\sqrt{3}/2$, then
\begin{equation}
 M_p(t)\leq\frac{1}{1-(d/3)t},
 \qquad 0\leq t<3/d.
 \label{eq:skellam-coefficient-envelope}
\end{equation}
\end{lemma}

\begin{proof}
The coefficient of $t^n$ on the right of
\eqref{eq:skellam-coefficient-envelope} is $(d/3)^n$. For $n=0$ and $n=1$,
the coefficients in \eqref{eq:skellam-coefficients} agree with the right side.
For $n=2$, the required inequality is
\begin{equation*}
 \frac{1}{12}\leq\frac{d^2}{9},
\end{equation*}
which is equivalent to $d^2\geq3/4$.

It remains to propagate the inequality. Within either parity class, the ratio
of two successive coefficients on the left, with indices differing by two, is
\begin{equation*}
 \frac{1}{(n+3)(n+4)}.
\end{equation*}
The corresponding ratio on the right is $(d/3)^2\geq1/12$. Starting from
$n=2$ for the even coefficients and from $n=1$ for the odd coefficients gives
the claim, because $1/((n+3)(n+4))\leq1/12$ for every $n\geq1$.
\end{proof}

Define
\begin{equation}
 p_- =\frac{2-\sqrt{3}}{4},
 \qquad
 p_+ =\frac{2+\sqrt{3}}{4}.
 \label{eq:skellam-thresholds}
\end{equation}

\begin{theorem}[Exact Skellam regimes]
\label{thm:skellam-phase}
For the centered Skellam law \eqref{eq:centered-skellam}, the right scale has
the following regimes:
\begin{equation}
 \cplus(X_{\lambda,p})=
 \begin{cases}
  0,&p=0,\\
  \displaystyle\max_{t>0}\Psi_p(t),&0<p<p_+,\\
  \displaystyle\frac{2p-1}{3},&p_+\leq p\leq1.
 \end{cases}
 \label{eq:skellam-cplus}
\end{equation}
For $0<p<p_+$, the maximum is attained at a positive finite argument and is
strictly larger than $\max\{0,(2p-1)/3\}$. No uniqueness is asserted. The left
scale is
\begin{equation}
 \cminus(X_{\lambda,p})=
 \begin{cases}
  \displaystyle\frac{1-2p}{3},&0\leq p\leq p_-,\\
  \displaystyle\max_{t>0}\Psi_{1-p}(t),&p_-<p<1,\\
  0,&p=1,
 \end{cases}
 \label{eq:skellam-cminus}
\end{equation}
and the middle regime is again interior controlled.
\end{theorem}

\begin{proof}
At $p=0$ there are no positive jumps, so $\cplus=0$ by
\Cref{thm:zero-pattern}. Suppose first that $p\geq p_+$. Then
$d\geq\sqrt{3}/2$, and \Cref{lem:skellam-coefficients} shows that the scale
$d/3$ is feasible. The local lower bound \eqref{eq:local-bounds}, or the limit
in \eqref{eq:skellam-Psi-series}, gives the reverse inequality. Hence
$\cplus=d/3$.

Now let $0<p<p_+$. Positive jumps are present, so $\cplus>0$. If $d\leq0$, the
limit of $\Psi_p(t)$ at zero is nonpositive and cannot control this positive
scale. If $0<d<\sqrt{3}/2$, the coefficient of $t$ in
\eqref{eq:skellam-Psi-series} is positive, so $\Psi_p(t)>d/3$ for all
sufficiently small positive $t$. Again the origin cannot control the scale.

The function $\Psi_p$ is continuous on $(0,\infty)$. Since $p>0$, the positive
exponential term in \eqref{eq:skellam-M} gives $M_p(t)\to\infty$, and hence
$\Psi_p(t)\to0$ as $t\to\infty$. Its positive supremum is therefore attained
away from both endpoints. This proves \eqref{eq:skellam-cplus}. Reflection
replaces $p$ by $1-p$, and $1-p_+=p_-$, which gives
\eqref{eq:skellam-cminus}.
\end{proof}

The phase theorem differs from the optimal sub-Poisson result of
\citet{leskela2026subpoisson}. Their variance proxy for the Skellam law is
$\lambda$ in both directions. Here the variance is already fixed at $\lambda$;
the optimized object is the sub-gamma scale, which detects the imbalance
parameter $p$ and changes its maximizing mechanism at $p_-$ and $p_+$.

For reference, \Cref{tab:skellam-numerics} gives numerical values obtained by
maximizing the explicit function \eqref{eq:skellam-objective}. The table is an
illustration, not part of the proof.

\begin{table}[ht]
\centering
\caption{Selected Skellam scales and right maximizing arguments.}
\label{tab:skellam-numerics}
\begin{tabular}{@{}ccc@{}}
\toprule
$p$ & $\cplus$ & right maximizing argument \\
\midrule
$0.001$ & $0.0656744$ & $14.0476$ \\
$0.010$ & $0.0818736$ & $10.9683$ \\
$0.250$ & $0.1382405$ & $5.6387$ \\
$0.500$ & $0.1742180$ & $3.8224$ \\
$0.750$ & $0.2204199$ & $2.0986$ \\
$p_+$ & $0.2886751$ & $0$ (local) \\
$1.000$ & $0.3333333$ & $0$ (local) \\
\bottomrule
\end{tabular}
\end{table}

%% file: sections/07_rare_jumps.tex
\section{Rare adverse-jump asymptotics}
\label{sec:rare}

The zero-pattern theorem is discontinuous at $p=0$: the right scale is zero
there and positive for every $p>0$. The next result describes how the positive
values approach zero. Let $t_p$ be any global maximizer of $\Psi_p$ in the
interior regime, and write
\begin{equation}
 L=\log(1/p).
 \label{eq:L-definition}
\end{equation}

\begin{theorem}[Rare positive jumps]
\label{thm:rare-jumps}
As $p\downarrow0$,
\begin{equation}
 t_p=L+3\log L-\log2+o(1),
 \label{eq:tp-asymptotic}
\end{equation}
and
\begin{equation}
 \cplus(X_{\lambda,p})\sim\frac{1}{L}.
 \label{eq:cp-asymptotic}
\end{equation}
The conclusion holds for every choice of a global interior maximizer; it does
not require or imply uniqueness.
\end{theorem}

\begin{proof}
It is convenient to write
\begin{equation}
 A(t)=\frac{2(e^t-1-t)}{t^2},
 \qquad
 M_p(t)=pA(t)+(1-p)A(-t).
 \label{eq:A-and-M}
\end{equation}
Consider the deterministic comparison point
\begin{equation*}
 s_p=L+3\log L-\log2.
\end{equation*}
Since $s_p\sim L$,
\begin{equation*}
 \frac{2pe^{s_p}}{s_p^2}=\frac{L^3}{s_p^2}\sim L,
 \qquad
 (1-p)A(-s_p)\sim\frac{2}{s_p}.
\end{equation*}
It follows that $M_p(s_p)\sim L$ and
\begin{equation}
 \Psi_p(s_p)\sim\frac{1}{L}.
 \label{eq:rare-lower}
\end{equation}
In particular, the optimal value is at least of this order.

We next show that $t_p\to\infty$. On every compact subinterval of
$(0,\infty)$, $M_p(t)$ converges uniformly to $A(-t)<1$, so $\Psi_p(t)$ is
eventually negative there. The expansion at zero is also uniform for small
$p$ and starts at $(2p-1)/3$, which stays bounded away from zero on the
negative side. This excludes both a bounded positive subsequence and convergence
to zero, in view of the positive lower bound \eqref{eq:rare-lower}.

At an interior maximizer, differentiation of
\eqref{eq:skellam-objective} gives the stationary equation
\begin{equation}
 t_p M_p'(t_p)=M_p(t_p)\{M_p(t_p)-1\}.
 \label{eq:stationary}
\end{equation}
Set
\begin{equation}
 q_p=\frac{2pe^{t_p}}{t_p^2}.
 \label{eq:q-definition}
\end{equation}
Because the maximum is positive, $M_p(t_p)>1$. Since
$A(-t)=2/t+O(t^{-2})$, this implies $q_p$ is bounded away from zero. As
$t_p\to\infty$, the exact decomposition
\begin{equation*}
 M_p(t)=\frac{2pe^t}{t^2}+\frac{2(1-2p)}{t}-\frac{2}{t^2}
          +\frac{2(1-p)e^{-t}}{t^2}
\end{equation*}
and its derivative give
\begin{align}
 M_p(t_p)&=q_p+\frac{2}{t_p}+o(1/t_p),
 \label{eq:M-rare-expansion}\\
 t_pM_p'(t_p)&=q_p(t_p-2)-\frac{2}{t_p}+o(1/t_p).
 \label{eq:Mp-rare-expansion}
\end{align}
Substitution in \eqref{eq:stationary}, followed by division by $q_p$, yields
\begin{equation}
 q_p=t_p-1+o(1),
 \label{eq:q-over-t}
\end{equation}
and hence $q_p/t_p\to1$.

Taking logarithms in \eqref{eq:q-definition} and using
\eqref{eq:q-over-t},
\begin{align*}
 t_p
 &=L+2\log t_p+\log q_p-\log2\\
 &=L+3\log t_p-\log2+o(1).
\end{align*}
This relation first gives $t_p/L\to1$, and then
$\log t_p=\log L+o(1)$. Equation \eqref{eq:tp-asymptotic} follows. Finally,
\eqref{eq:M-rare-expansion} and \eqref{eq:q-over-t} show that
$M_p(t_p)\sim t_p$. Therefore
\begin{equation*}
 \cplus(X_{\lambda,p})
 =\frac{1-M_p(t_p)^{-1}}{t_p}
 \sim\frac{1}{t_p}\sim\frac{1}{L},
\end{equation*}
which proves \eqref{eq:cp-asymptotic}.
\end{proof}

The maximizer moves to order $\log(1/p)$. At that scale, the exponentially
amplified rare positive component and the $t^{-2}$ Beta-remainder factor
balance at order $t$.

%% file: sections/08_discussion.tex
\section{Discussion}
\label{sec:discussion}

For a general finite-variance infinitely divisible law, both optimal
true-variance sub-gamma scales are functions of the same signed canonical
remainder. Together with $V$, the full remainder law recovers the L\'evy
triplet, while the two-number summary $(\cplus,\cminus)$ identifies which jump
directions are present.
The two scales coincide only when the relevant MGFs happen to impose the same
optimization, not because the definition is intrinsically symmetric.

Adding nonpositive jumps can only reduce the optimal right scale, even though it
increases the true variance. If the positive MGF boundary already fixes the
right pole, that reduction is exactly offset and the scale is unchanged. The
bilateral Gamma calculation is a clean instance: the negative Gamma component
does not alter the right boundary, and the positive component does not alter the
left boundary.

The Skellam family shows why local cumulants and MGF boundaries are not enough.
Its MGF is entire, so there is no finite boundary obstruction. The third
cumulant gives the local candidate $(2p-1)/3$, but this candidate is globally
valid only beyond $p_+$. Below that threshold, the optimizer moves into the
interior. Reflection creates the second threshold $p_-$. At $p=1/2$, the third
cumulant is zero although both optimal scales are positive and interior
controlled. It is an exact example in which the whole signed MGF, rather than a
finite cumulant list, determines the scales.

We do not assert uniqueness of an interior optimizer. The signed remainder does
not supply a scalar stochastic order that
would replace the M-class criterion from the spectrally positive setting. The
results are one-dimensional; a multivariate extension would require a family of
directional canonical measures and cannot be summarized by two numbers. The
symbolic, numerical, and Lean files check selected algebraic and
optimization steps. They do not replace the analytic proofs or formalize the
measure-theoretic L\'evy representation.